\documentclass[11pt,reqno]{amsart}

\usepackage{amsmath, amsfonts, amsthm, amssymb, color}

\newtheorem{theo}{Theorem}[section]
\newtheorem{lemm}{Lemma}[section]

\newtheorem{coro}{Corollary}[section]

\theoremstyle{definition}
\newtheorem{defi}{Definition}[section]

\theoremstyle{remark}
\newtheorem{rem}{Remark}[section]

\numberwithin{equation}{section}

\def\u{ u}

\newcommand{\M}{\mathbb M}
\newcommand{\D}{\mathbb D}
\renewcommand{\S}{{\mathbb S}_\nu}
\newcommand{\Sk}{{\mathbb S}_\kappa}

\newcommand{\R}{{\mathbb R}}
\newcommand{\Dv}{{\rm div}}
\newcommand{\eps}{{\varepsilon}}
\newcommand{\vfi}{{\varphi}}
\newcommand{\vfin}{{\varphi_n}}

\newcommand{\E}{{\mathcal E}_0}
\newcommand{\Er}{{\mathcal E}_r}

\newcommand{\T}{{\mathbb T}_\nu}
\def\f{\frac}
\renewcommand{\O}{\Omega}

\def\hf1{^\f{1}{1-\xi^2}}

\def\be{\begin{equation}}
\def\en{\end{equation}}
\def\bs{\begin{split}}
\def\es{\end{split}}

\renewcommand{\Mc}{M_0(\sro_0,\sqrt{\rho_{0}} u_0, f,\M)}
\newcommand{\Mcr}{M_r(\sro_0,\sqrt{\rho_{0}} u_0, f,\M)}
\newcommand{\Mcn}{M_0(\sqrt{\rho_{0,n}},\sqrt{\rho_{0,n}}u_{0,n} ,f_n,\M_n)}
\newcommand{\Mcnr}{M_{r_{n}}(\sro_{0,n},\sqrt{\rho_{0,n}} ,f_n,\M_n)}
\newcommand{\Cs}{\overline {C}_*}

\newcommand{\Cin}{C^\infty_c(\R^+\times\O)}
\newcommand{\domt}{{(\R^+;L^2(\O))}}
\newcommand{\domf}{{(\R^+;L^4(\O))}}
\newcommand{\domo}{{(\R^+;L^1(\O))}}
\newcommand{\sro}{{\sqrt{\rho}}}
\newcommand{\srof}{\rho^{1/4}}

\newcommand{\F}{{F}}

\newcommand{\De}{\mathcal{D}_E}
\newcommand{\DE}{\mathcal{D}_\mathcal{E}^0}
\newcommand{\DEr}{\mathcal{D}_\mathcal{E}^r}

\newcommand{\A}{\mathbb{A}}

\newcommand{\In}[1]{\textcolor{cyan}{#1}}

\author{Ingrid Lacroix-Violet}
\address{Laboratoire Paul Painlev\'e, Universit\'e de Lille 1}
\email{ingrid.violet@math.univ-lille1.fr}
\author{Alexis F. Vasseur}
\address{Department of Mathematics,
The University of Texas at Austin.}
\email{vasseur@math.utexas.edu}

\title[Global  solutions  to the Quantum Navier-Stokes equation]
{Global Weak Solutions to the Compressible Quantum Navier-Stokes Equation and its semi-classical limit}

\thanks{\textbf{Acknowledgment.} A. F. Vasseur was partially supported by the NSF Grant DMS 1209420. 
}
\date{\today}

\begin{document}

\begin{abstract}
This paper is dedicated to the construction of global weak solutions to the quantum Navier-Stokes equation, for any initial value with bounded energy and entropy. 
The construction is uniform with respect to the Planck constant. This allows to perform the semi-classical limit to the associated compressible Navier-Stokes equation.
One of the difficulty of the problem is to deal with the degenerate  viscosity, together with the lack of integrability on the velocity. Our method is  based on the construction of weak solutions that are renormalized in the velocity variable. The existence, and stability of these solutions do not need the Mellet-Vasseur inequality.  
\end{abstract}

\maketitle

\section{Introduction}
Quantum models can be used to describe superfluids \cite{LoMo93}, quantum semiconductors \cite{FeZhou93}, weakly interacting Bose gases \cite{Grant73} and quantum trajectories of Bohmian mechanics \cite{Wyatt05}. They have attracted considerable attention in the last decades due, for example, to the development of nanotechnology applications.

In this paper, we consider the barotropic compressible quantum Navier-Stokes equations, which has been derived in \cite{BrMe2010}, under some assumptions, using a Chapman-Enskog expansion in Wigner equation. In particular, we are interested in the existence of global weak solutions together with the associated  semi-classical limit.  The quantum Navier-Stokes equation that we are considering read as:
\begin{equation}
\begin{split}
\label{eq_system}
&\rho_t+\Dv(\rho\u)=0,
\\&(\rho\u)_t+\Dv(\rho\u\otimes\u)+\nabla\rho^{\gamma}-2 \Dv(\sqrt{\nu\rho}\mathbb{S}_\nu+\sqrt{\kappa \rho}\mathbb{S}_\kappa)=\sro f + \sqrt{\kappa} \Dv(\sro \M),
\end{split}
\end{equation}
where 
\begin{equation}\label{eq_formal}
\sqrt{\nu\rho}\mathbb{S}_\nu= \rho\mathbb{D}\u,\qquad \Dv(\sqrt{\kappa \rho}\mathbb{S}_\kappa)=\kappa\rho\nabla\left(\frac{\Delta\sqrt{\rho}}{\sqrt{\rho}}\right),
\end{equation}
and with initial data
\begin{equation}
\label{initial data}
\rho(0,x)=\rho_0(x),\;\;\;(\rho\u)(0,x)=(\rho_0\u_0)(x)\quad\text{ in } \O,
\end{equation}
where $\rho$ is the density, $\gamma>1$, $\u\otimes\u$ is the matrix with components $\u_i\u_j,$ $\mathbb{D}\u=\frac{1}{2}\left(\nabla\u+\nabla\u^T\right)$ is the symetric part of the velocity gradient,
 and $\O=\mathbb{T}^d$ is the $d-$dimensional torus, here $d=2$ or 3. The vector valued function $f$, and the matrix valued function $\M$ are source terms. 

The relation (\ref{eq_formal}) between the stress tensors and the solution $(\sqrt{\rho}, \sro u)$  will be proved in the following form. For the quantic part, it will be showed that 
\begin{equation}\label{eq_quantic}
2\sqrt{\kappa\rho}\mathbb{S}_\kappa=2\kappa \left(\sqrt{\rho}\left(\nabla^2\sqrt{\rho}-4(\nabla\srof\otimes\nabla\srof)\right)\right).
\end{equation}
For the viscous term, the matrix valued function $\mathbb{S}_\nu$ is the symmetric part of a matrix valued function $\T$, where 
\begin{equation}\label{eq_viscous}
\sqrt{\nu\rho}\T=\nu\nabla(\rho u)-2\nu\sqrt{\rho} u\cdot\nabla\sqrt{\rho}.
\end{equation}
Whenever, $\rho$ is regular and away from zero, the quantic part of (\ref{eq_formal})    is equivalent to (\ref{eq_quantic}), and the matrix function $\T$ is formally $\sqrt{\nu\rho}\nabla u$. However, the a priori estimates do not allow to define $1/\sro$ and $\nabla u$. 
\vskip0.3cm

The energy of the system is given by 
$$
E(\sro,\sro u) =\int_{\O} \left(\rho\frac{|u|^2}{2}+\frac{\rho^\gamma}{\gamma-1}+2\kappa|\nabla\sro|^2\right)\,dx,
$$
with dissipation of entropy (in the case without source term)
$$
\De(\mathbb{S}_\nu)=2\int_{\O}|\mathbb{S}_\nu|^2\,dx,
$$
which is formally:
$$
2\nu\int_{\O}\rho|\mathbb{D}u|^2\,dx.
$$

In \cite{BrDe2003,BrDeLi2003}, Bresch and Desjardins introduced a new entropy of the system, now known as the BD entropy:
$$
\mathcal{E}_{BD}(\rho,u)=\int_\O \left(\rho \frac{|u+\nu\nabla\ln\rho|^2}{2}+\frac{\rho^\gamma}{\gamma-1}+2\kappa|\nabla\sro|^2\right)\,dx,
$$
with associated dissipation (again without source term)
$$
\mathcal{D}_{BD} (\rho,u)=\nu \int_\O \left(\frac{4}{\gamma}|\nabla \rho^{\gamma/2}|^2+\kappa\rho|\nabla^2\ln\rho|^2+2\rho |\A u|^2\right)\,dx,
$$
where $\A$ is the antisymmetric part of the matrix $\nabla u$.
The function $\ln \rho$ is not controled by the a priori estimates. But we can  use two other quantities, $\E$ and $\DE$, which are defined as follows:
\begin{eqnarray*}
&& \E(\sro, \sro u)=\int_{\O}\left(\rho\frac{|u|^2}{2}+(2\kappa+4\nu^2)|\nabla\sro|^2+\rho^\gamma\right)\,dx,\\
&&\DE(\sro,\sro u)=\int_{\O}\left(\nu  |\nabla \rho^{\gamma/2}|^2+\nu \kappa\left(|\nabla\srof|^4+|\nabla^2\sro|^2\right)+|\T|^2  \right)\,dx.
\end{eqnarray*}
From J\"ungel \cite{Jue10} the functions $\E$ and $\DE$, are equivalent, respectively, to $E(\sro,\sro u)+\mathcal{E}_{BD}(\sro,\sro u)$ and $\De(\sro,\sro u,\S)+\mathcal{D}_{BD}(\rho, u)$, whenever each term can be defined, and  $\S$ is the symmetric part of $\T$ with $\T=\sqrt{\nu\rho}\nabla u$. Namely there exists a universal constant $C_*$ such that for any such $(\rho, u)$:
\begin{eqnarray*}
&&\frac{1}{C_*}(E(\sro, \sro u)+\mathcal{E}_{BD}(\rho,u))\leq \E(\sro,\sro u)\leq C_*(E(\sro, \sro u)+\mathcal{E}_{BD}(\rho,u)),\\
&&\frac{1}{C_*}(\De(\S)+\mathcal{D}_{BD}(\rho,u))\leq \DE(\sro,\sro u,\T)\leq C_*(De(\S)+\mathcal{D}_{BD}(\rho,u)).
\end{eqnarray*}
 
\vskip0.3cm

The aim of this paper is to construct weak solutions for the system (\ref{eq_system}) using the a priori estimates provided by the energy and BD entropy inequalities. The main idea is to introduce a slightly stronger notion of weak solution that we call renormalized solutions. They are defined in the following way. For any function $\vfi\in W^{2,\infty}(\R^d)$, there exists two measures $\overline{R}_\vfi, R_{\vfi}\in\mathcal{M}(\R^+\times \O)$ such that the following is verified in the sense of distribution:
\begin{equation}\label{eq_renormalise}
\begin{split}
&\partial_t(\rho\vfi(u))+\Dv(\rho u \vfi(u))+\vfi'(u)\cdot\nabla\rho^\gamma\\
&\qquad\qquad -2\Dv(\sro\vfi'(u)(\sqrt{\nu}\S+\sqrt{\kappa}\Sk))\\
&\qquad =\sro\vfi'(u)\cdot f+\sqrt{\kappa}\Dv(\sro\vfi'(u)\M)+R_\vfi, 
\end{split}
\end{equation}
with $\Sk$ verifies (\ref{eq_quantic}), and $\S$ is the symmetric part of $\T$ such that for every $i,j,k$ between 1 and $d$:
\begin{equation}\label{eq_viscous_renormalise}
\sqrt{\nu\rho}\vfi_i'(u)[\T]_{jk}=\nu \partial_j(\rho\vfi'_i(u)u_k)-2\In{\nu}\sro u_k\vfi'_i(u)\partial_j\sro+ \overline{R}_\vfi,
\end{equation}
and
$$
\|R_\vfi\|_{\mathcal{M}(\R^+\times\O)}+\|\overline{R}_\vfi\|_{\mathcal{M}(\R^+\times\O)}\leq C\|\vfi''\|_{L^\infty}.
$$
The precise definition of weak solutions and renormalized weak solutions is laid out in definitions \ref{def_weak} and \ref{def_renormalise}.
Note that, taking a sequence of function $\vfi_n$ such that $\vfi_n(y)$ converges to $y_i$, but $\|\vfi''\|_{L^\infty}$ converges to 0, we can retrieve formally the equation (\ref{eq_system}). We will show that it is actually true that any renormalized weak solution is, indeed, a weak solution. 
\vskip0.3cm

For every 
$$
f\in \left[L^1\domt
\right]^d,\qquad \M\in\left[L^2\domt\right]^{d\times d},
$$
and every $(\sqrt{\rho_0}, \sqrt{\rho_0} u_0)$ such that $ \E(\sqrt{\rho_0}, \sqrt{\rho_0} u_0)$ is bounded,
we define 

\begin{equation}
\Mc=\E(\sqrt{\rho_0}, \sqrt{\rho_0} u)+\|f\|_{L^1\domt}+\|\M\|_{L^2\domt}.
\end{equation}

The main theorem of the paper is the following.

\begin{theo} \label{theo_main}
\begin{enumerate}
\item There exists a universal constant $\Cs>0$ such that the following is true.
Let $\sqrt{\rho_0}, \sqrt{\rho_0} u_0, f, \M$ be such that $\Mc$ is bounded.
Then, for any $\kappa\geq 0$, there exists a renormalized solution $(\sro,\sro u)$ of (\ref{eq_system}) with
\begin{eqnarray*}
&&\E(\sro(t), \sro(t) u(t))\leq \Mc,\qquad\mathrm{for}\ \ t>0,\\
&& \int_0^\infty\DE(\sro(t), \sro(t) u(t))\,dt\leq \Mc. 
\end{eqnarray*}
Moreover, for every $\vfi\in W^{2,\infty}(\R^d)$, 
\begin{eqnarray*}
&&\qquad \|R_\vfi\|_{\mathcal{M}(\R^+\times\O)}+\|\overline{R}_\vfi\|_{\mathcal{M}(\R^+\times\O)}\leq \Cs \|\vfi''\|_{L^\infty}\Mc.
\end{eqnarray*}
Moreover, $\rho\in C^0(\R^+;L^p(\O))$ for $1\leq p<\sup(3,\gamma)$, and $\rho u\in C^0(\R^+; L^{3/2}(\O)-weak)\cap  C^0(\R^+; L^{\frac{2\gamma}{\gamma+1}}(\O)-weak)$. 
\item Any renormalized solution  of \eqref{eq_system} is a weak solution of \eqref{eq_system} with the same initial value.
\item Consider any sequences $\kappa_n\geq0$,   converging to $\kappa\geq0$, $\nu_n>0$  converging to $\nu>0$,  $(\sqrt{\rho_{0,n}},\sqrt{\rho_{0,n}} u_{0,n}, f_n,\M_n)$ such that  $\Mcn$ is uniformly bounded, and an associated weak renormalized solution $(\sqrt{\rho_n}, \sqrt{\rho_n} u_n)$ to  \eqref{eq_system}. Then, there exists a subsequence (still denoted with $n$),  and $(\sro, \sro u)$ renormalized solution to  \eqref{eq_system} with initial value $(\sqrt{\rho_0}, \sqrt{\rho_0}u_0)$ and Planck constant $\kappa$, such that 
$\rho_n$ converges to $\rho$ in $ C^0(\R^+;L^p(\O))$ for $1\leq p<\sup(3,\gamma)$, and $\rho_n u_n$ converges to $\rho u$ in $C^0(\R^+; L^{3/2}(\O)-weak)\cap  C^0(\R^+; L^{\frac{2\gamma}{\gamma+1}}(\O)-weak)$. The function ${\T}_{,n}$ converges weakly in $L^2(\R^+\times\O)$ to $\T$. Moreover, for every function $\vfi\in W^{2,\infty}(\R^d)$, $\sqrt{\rho_n}\vfi(u_n)$ converges strongly in $L_{\mathrm{loc}}^p(\R^+\times\O)$ to $\sro \vfi(u)$ for $1\leq p< 6$.

\end{enumerate}

\end{theo}

Note that all the results hold for any values of $\kappa$, including  the Navier-Stokes case $\kappa=0$. For $\kappa>0$ we can have from the a priori estimates, better controls on the solutions,  and convergence in stronger norms. The stability part of the result  includes the follwoing case of the semi-classical limits $0<\kappa_n\to 0$.   

\begin{coro} The semi-classical limit. 
Consider    $(\sqrt{\rho_{0}},\sqrt{\rho_{0}} u_0, f,\M)$ such that  the quantity $\Mc$ is bounded, and consider an associated weak renormalized solution $(\sqrt{\rho_\kappa}, \sqrt{\rho_\kappa} u_\kappa)$ to the quantum Navier-Stokes equations \eqref{eq_system} with $\kappa>0$.Then, there exists a subsequence (still denoted with $\kappa$),   and $(\sro, \sro u)$ renormalized solution to  the Navier-Stokes equations (\eqref{eq_system} with $\kappa=0$) with same initial value such that 
$\rho_\kappa$ converges to $\rho$ in $ C^0(\R^+;L^p(\O))$ for $1\leq p<\sup(3,\gamma)$, and $\rho_\kappa u_\kappa$ converges to $\rho u$ in $C^0(\R^+; L^{3/2}(\O)-weak)\cap  C^0(\R^+; L^{\frac{2\gamma}{\gamma+1}}(\O)-weak)$. The function ${\T}_\kappa$ converges weakly in $L^2(\R^+\times\O)$ to $\T$. Moreover, for every function $\vfi\in W^{2,\infty}(\R^d)$, $\sqrt{\rho_\kappa}\vfi(u_\kappa)$ converges strongly in $L_{\mathrm{loc}}^p(\R^+\times\O)$ to $\sro \vfi(u)$ for $1\leq p< 6$.
\end{coro}


For this problem, the a priori estimates include control on the gradient of some density quantities. This provides compactness on both the density $\rho$ and the momentum $\rho u$. The difficulty is due to the fact that we have only a control of $\rho u^2$ in $L^\infty(\R^+,L^1(\O))$.
This cannot prevent concentration phenomena in the construction of solutions in the term $\rho u\otimes u$. When $\kappa=0$, the same problem arises for the term $\sro u\nabla\sro$, since the only a priori estimates available on both $\sro u$ and $\nabla \sro$ are in $L^\infty(\R^+,L^2(\O))$.  
\vskip0.3cm
This problem can be avoided by the introduction of additional terms as drag forces or cold pressure, as proposed by Bresch and Desjardins \cite{BrDe2007}. In the case of drag forces, the system is 
\begin{equation}
\begin{split}
\label{eq_system_r}
&\rho_t+\Dv(\rho\u)=0,
\\&(\rho\u)_t+\Dv(\rho\u\otimes\u)+\nabla\rho^{\gamma}-2 \Dv(\sqrt{\nu\rho}\mathbb{S}_\nu+\sqrt{\kappa \rho}\mathbb{S}_\kappa)\\
&\qquad\qquad\qquad=-r_0u-r_1\rho |u|^2u+\sro f + \sqrt{\kappa} \Dv(\sro \M),
\end{split}
\end{equation}
still endowed with \eqref{eq_viscous} and \eqref{eq_quantic}. The solutions of this  kind of augmented systems can be explicitly  constructed via a Faedo-Galerkin method. This was first performed by Zatorska in \cite{zatorska2012} for the classical case ($\kappa=0$) with chemical reactions and where the drag forces were replaced by cold pressure terms as $\nabla (\rho^{-N})$, for $N$ big enough. In the quantum case solutions have been constructed in \cite{GiLaVi2015} in the case with cold pressure, and in \cite{VasseurYu2015} in the case of drag forces as \eqref{eq_system_r}. 
\vskip0.3cm
When considering the system without additional terms (as drag forces or cold pressure), it has been shown  in \cite{MelletVasseur2007} that, in the classical case $\kappa=0$, solutions of \eqref{eq_system} verify formally that 
\begin{equation}\label{eq_log}
\int_{\O} \rho |u|^2\ln(1+|u|^2)\,dx
\end{equation}
is uniformly bounded in time, provided that the bound is valid at $t=0$. It was also shown that a sequence of solutions, such that this quantity is uniformly bounded at $t=0$ (together with the energy and BD entropy), will converge, up to a subsequence, to a solution of the Navier-Stokes equation.
\vskip0.3cm
The standard way to construct weak solutions of systems verifying the weak stability is to construct solutions to an approximated problem (as the Faedo Galerkin method) for which the a priori estimates are still uniformly valid. Usually, those solutions are smooth so that every formal computation is actually true. However, in this context, the approximated problem needs to be compatible with the usual energy, the BD entropy, and the additional mathematical inequality \eqref{eq_log}. Only recently, such an approximated problem has been found \cite{LiXin2015} in dimension two, and in dimension three with unphysical stress tensors. Moreover, the regularity of the associated solutions is limited (not $C^\infty$).  For solutions constructed via a Faedo-Galerkin method, the energy and BD entropy estimates can be verified at the approximated level, since $u$  is an admissible test function. This is not the case for the mathematical inequality \eqref{eq_log}.  In \cite{VasseurYu2015-2}, the construction of solutions to \eqref{eq_system} with $\kappa=0$ was obtained following a different strategy. It was shown that  limits of solutions to \eqref{eq_system_r}  when $\kappa$ converges to 0 verify \eqref{eq_log}, even if it is not verified for $\kappa>0$. The idea was to show that the quantity
$$
\int_\O \rho\vfi(|u|)\,dx
$$
are uniformly bounded for smooth and bounded functions $\vfi$. This allowed to recover \eqref{eq_log} for $\kappa=0$ thanks to a sequence of approximations $\vfi_n$ of the function $y\to y^2\ln(1+y^2)$.
\vskip0.3cm
Spririto and Antonelli showed in \cite{AntSpi2016} that, formally, an estimate on \eqref{eq_log} can still be obtained on solutions to \eqref{eq_system} for a range of $\kappa$ close (or bigger) than $\nu$. But this estimate cannot be used for the semi-classical limit. 
\vskip0.3cm
The notion of renormalized solutions is inspired from \cite{VasseurYu2015-2}. However, this notion is not used to recover an estimate on \eqref{eq_log}, which is known to be not verified for some range of $\kappa$. The idea is that we can obtain the stability of renormalized solutions, since the notion avoids the problem of concentration. Consider, for instance, the term
$$
\rho_n u_n \vfi(u_n).
$$
Since $\rho_n$ and $\rho_n u_n$ are compact in $L^p$, for a $p>1$, we can show that, up to a subsequence, $\rho_n$ converges almost everywhere to a function $\rho$, and $u_n$ converges almost everywhere on $\{(t,x) \ | \ \ \rho(t,x)>0\}$ to a function $u$. Hence $\rho_n u_n \vfi(u_n)$ converges almost everywhere to $\rho u\vfi(u)$. The function $\vfi$ prevent concentration, and so $\rho_n u_n \vfi(u_n)$ converges strongly to $\rho u\vfi(u)$. 
\vskip0.3cm
The challenge is then to show that the renormalized solutions, are indeed, weak solutions in the general sense (see Definition \ref{def_weak}).
It is obtained by considering a sequence of bounded functions $\vfi_n$, uniformly dominated by $y\to|y|$, and  converging almost everywhere to $y \to y_i$, for a fixed direction $i$. This provides  the momentum equation for $\rho u_i$ at the limit $n\to \infty$. The key point, is that, while performing this limit, the functions $\rho, u$ are fixed. Considering, for example, the term 
$$
\rho u \vfi_n(u).
$$
The function $\vfi_n(u)$ converges almost everywhere to $u_i$. And, thanks to the Lebesgue's dominated convergence Theorem, $\rho u\vfi_n(u)$ converges in $L^1$ to $\rho u u_i$. Note that the boundedness of $\rho u u_i$ in $L^1$ is enough for this procedure.  Choosing the sequence of $\vfi_n$ such that $\|\vfi_n''\|_{L^\infty}$ converges to 0, we show that the extra terms $R_{\vfi_n}$ and $\overline{R}_{\vfi_n}$ converge to 0 when $n$ converges to $\infty$. 
\vskip0.3cm
The main difference with \cite{VasseurYu2015-2} is that we do not need to reconstruct the energy inequality nor the control on \eqref{eq_log} via the sequence of functions $\vfi_n$. Hence, we do not need an explicit form of the terms involving second derivatives of $\vfi$ in the definition of renormalized solutions. Those terms (for which we do not have stability) are dumped in the extra terms $R_\vfi$ and $\overline{R}_\vfi$.

\section{Preliminary results and main ideas}

We are first working on the System \eqref{eq_system_r} with drag forces. The definitions will be valid for all the range of parameter, $r_0\geq0, r_1\geq0, \kappa\geq0, \nu>0$. The energy and the BD entropy on solutions to \eqref{eq_system_r} provide controls on 
\begin{eqnarray*}
&& \Er(\sro, \sro u)=\int_{\O}\left(\rho\frac{|u|^2}{2}+(2\kappa+4\nu^2)|\nabla\sro|^2+\rho^\gamma+r_0(\rho-\ln\rho)\right)\,dx,\\
&&\DEr(\sro,\sro u)=\\
&&\qquad\qquad\int_{\O}\left(\nu  |\nabla \rho^{\gamma/2}|^2+\nu \kappa\left(|\nabla\srof|^4+|\nabla^2\sro|^2\right)+|\T|^2  +r_0 |u|^2+r_1\rho |u|^4\right)\,dx.
\end{eqnarray*}
From these quantities, we can obtain the following a priori estimates. For the sake of completeness we show how to obtain them in the appendix.
\begin{equation}
\begin{split}
\label{eq_a_priori}
&\sro\in L^\infty\domt, \qquad \nabla\sro\in L^\infty\domt, \qquad \nabla \rho^{\gamma/2}\in L^2\domt\\ 
&\sro u \in L^\infty\domt, \qquad \T\in L^2\domt, \qquad \sqrt{\kappa}\nabla^2\sro\in L^2\domt,\\
&\kappa^{1/4}\nabla \srof\in L^4\domf, \qquad  r_1^{1/4}\srof u\in L^4\domf,\\
&r_0^{1/2} u\in L^2\domt, \qquad r_0\ln\rho\in L^\infty\domo.
\end{split}
\end{equation}
Note that those a priori estimates are not sufficient to define $\nabla u$ as a function. The statement that $\sro \nabla u$ is bounded in $L^2$ means that there exists a function $\T\in L^2\domt$ such that: 
$$
\sqrt{\nu}\sro \T=\Dv(\rho u)-\sro u\cdot \nabla\sro,
$$
which is, formally, $\rho\nabla u$.
The definition of weak solutions and renormalized weak solutions for the system \eqref{eq_system_r} are as follows.
\begin{defi}\label{def_weak}
We say that $(\sro,\sro u)$ is a weak  solution to \eqref{eq_system_r}, if   it verifies the a priori estimates \eqref{eq_a_priori}, and  
for any function $\psi\in \Cin$:
\begin{eqnarray*}
&&\int_0^\infty\int_\O \left(\rho \psi_t +\rho  u\cdot \nabla\psi \right)dx\, dt=0,\\
&&\int_0^\infty\int_\O \left(\rho u \psi_t +(\rho u\otimes u-{2}\sqrt{\nu\rho} \S-2\sqrt{\kappa\rho}\Sk-\sqrt{\kappa \rho} \M)\cdot \nabla\psi+\F \psi \right)dx\, dt=0,
\end{eqnarray*}
with $\S$ the symmetric part of $\T$ verifying \eqref{eq_viscous}, $\Sk$ verifying \eqref{eq_quantic}, and  
\begin{equation}\label{eq_F}
F=-2\rho^{\gamma/2}\nabla\rho^{\gamma/2}-r_0 u-r_1\rho|u|^2u +\sro f,
\end{equation}
and  for any $\overline{\psi}\in C^\infty_c(\R)$:
\begin{eqnarray*}
&&\lim_{t\to0}\int_\O \rho(t,x)\overline{\psi}(x)\,dx=\int_\O \rho_0(x)\overline{\psi}(x)\,dx,\\
&&\lim_{t\to0}\int_\O \rho(t,x)u(t,x)\overline{\psi}(x)\,dx=\int_\O \rho_0(x)u_0(x)\overline{\psi}(x)\,dx.
\end{eqnarray*}

\end{defi}

\begin{defi}\label{def_renormalise}
We say that $(\sro,\sro u)$ is a  renormalized weak  solution to (\ref{eq_system_r}), if  it verifies the a priori estimates (\ref{eq_a_priori}), and  
 for  any function $\vfi\in W^{2,\infty}(\R^d)$, there exists two measures $R_{\vfi}, \overline{R}_\vfi\in \mathcal{M}(\R^+\times\O)$, with 
 $$
 \|R_{\vfi}\|_{ \mathcal{M}(\R^+\times\O)}+ \|\overline{R}_{\vfi}\|_{ \mathcal{M}(\R^+\times\O)}\leq C \|\vfi''\|_{L^\infty(\R)},
 $$
 where the constant $C$ depends only on the solution $(\sro,\sro u)$, and 
 for any function $\psi\in \Cin$, 
\begin{eqnarray*}
&&\int_0^\infty\int_\O \left(\rho \psi_t +\rho  u\cdot \nabla\psi \right)dx\, dt=0,\\
&&\int_0^\infty\int_\O \left(\rho \vfi(u) \psi_t +\left(\rho \vfi(u) u-({2}\sqrt{\nu\rho} \S{+2}\sqrt{\kappa\rho}\Sk{+}\sqrt{\kappa \rho} \M) \vfi'(u)\right)\cdot \nabla\psi\right.\\
&&\qquad\qquad\qquad\qquad\qquad\qquad\qquad\left.+\F\cdot \vfi'(u) \psi \right)dx\, dt=\left \langle R_{\vfi}, \psi\right\rangle,
\end{eqnarray*}
with $\S$ the symmetric part of $\T$ verifying \eqref{eq_viscous_renormalise}, $\Sk$ verifying \eqref{eq_quantic}, {$f$ given by \eqref{eq_F},} 
and  for any $\overline{\psi}\in C^\infty_c(\R)$:
\begin{eqnarray*}
&&\lim_{t\to0}\int_\O \rho(t,x)\overline{\psi}(x)\,dx=\int_\O \rho_0(x)\overline{\psi}(x)\,dx,\\
&&\lim_{t\to0}\int_\O \rho(t,x)u(t,x)\overline{\psi}(x)\,dx=\int_\O \rho_0(x)u_0(x)\overline{\psi}(x)\,dx.
\end{eqnarray*}

\end{defi}
Let us  define 
$$
\Mcr=\Er(\sqrt{\rho_0}, \sqrt{\rho_0} u_0)+\|f\|_{L^1\domt}+\|\M\|_{L^2\domt}.
$$

The main theorem proved in this paper is the following. 
\begin{theo} \label{theo_df}
\begin{enumerate}
\item There exists a universal constant $\Cs>0$ such that the following is true.
Let $\sqrt{\rho_0}, \sqrt{\rho_0} u_0, f, \M$ be such that $\Mcr$ is bounded.
Then, for any $\kappa\geq 0$, $r_0\geq0$, $r_1\geq0$, there exists a renormalized solution $(\sro,\sro u)$ of (\ref{eq_system_r}) with
\begin{eqnarray*}
&&\Er(\sro(t), \sro(t) u(t))\leq \Mcr,\qquad\mathrm{for}\ \ t>0,\\
&& \int_0^\infty\DEr(\sro(t), \sro(t) u(t))\,dt\leq \Mcr. 
\end{eqnarray*}
Moreover, for every $\vfi\in W^{2,\infty}(\R^d)$, 
\begin{eqnarray*}
&&\qquad \|R_\vfi\|_{\mathcal{M}(\R^+\times\O)}+\|\overline{R}_\vfi\|_{\mathcal{M}(\R^+\times\O)}\leq \Cs \|\vfi''\|_{L^\infty}\Mc.
\end{eqnarray*}
Moreover, $\rho\in C^0(\R^+;L^p(\O))$ for $1\leq p<\sup(3,\gamma)$, and $\rho u\in C^0(\R^+; L^{3/2}(\O)-weak)\cap  C^0(\R^+; L^{\frac{2\gamma}{\gamma+1}}(\O)-weak)$. 
\item Any renormalized solution  of \eqref{eq_system_r} is a weak solution of \eqref{eq_system_r} with the same initial value.
\item If $r_0>0$, $r_1>0$, and $\kappa>0$, then any weak solution to \eqref{eq_system_r} is also a renormalized solution to \eqref{eq_system_r}  with the same initial value.
\item Consider any sequences $\kappa_n\geq0$, $r_{0,n}\geq0$, $r_{1,n}\geq0$,  $\nu_n>0$,   converging respectively to $\kappa\geq0$,  $r_0\geq0$, $r_1\geq0$, and  $\nu>0$,  $(\sqrt{\rho_{0,n}}),\sqrt{\rho_{0,n}} u_{0,n}, f_n,\M_n)$ such that  $\Mcnr$ is uniformly bounded, and an associated weak renormalized solution $(\sqrt{\rho_n}, \sqrt{\rho_n} u_n)$ to  \eqref{eq_system_r}. Then, there exists a subsequence (still denoted with $n$),  and $(\sro, \sro u)$ renormalized solution to  \eqref{eq_system_r} with initial value $(\sqrt{\rho_0}, \sqrt{\rho_0}u_0)$  Planck constant $\kappa$, and drag forces coefficients $r_0,r_1$ such that 
$\rho_n$ converges to $\rho$ in $ C^0(\R^+;L^p_\mathrm{loc}(\O))$ for $1\leq p<\sup(3,\gamma)$, and $\rho_n u_n$ converges to $\rho u$ in $C^0(\R^+; L^{3/2}(\O)-weak)\cap  C^0(\R^+; L^{\frac{2\gamma}{\gamma+1}}(\O)-weak)$. The function ${\T}_{,n}$ converges weakly in $L^2(\R^+\times\O)$ to $\T$. Moreover, for every function $\vfi\in W^{2,\infty}(\R^d)$, $\sqrt{\rho_n}\vfi(u_n)$ converges strongly in $L_{\mathrm{loc}}^p(\R^+\times\O)$ to $\sro \vfi(u)$ for $1\leq p< 10\gamma/3$.

\end{enumerate}
\end{theo}

\begin{rem}
We can actually show in (1) of the previous theorem that there is one solution verifying 
$$
\overline{R}_{\vfi,i,k,j}=\sum_{l=1}^d \sro u_j \vfi''_{i,l}(u){\T}_{k,l}.
$$
But this is not needed to show that a renormalized solution is a weak solution. We cannot do the same for the term $R_\vfi$.
\end{rem}

Note that Theorem \ref{theo_df} together with \cite{VasseurYu2015} implies Theorem \ref{theo_main}. Indeed, \cite{VasseurYu2015} provides the  construction of weak solutions to \eqref{eq_system_r} with positive $r_0, r_1,\kappa$.  Part (3) in Theorem \ref{theo_df} insures that this solution is actually a renormalized solution.  Considering sequences $r_{0,n}\geq0$ and $r_{1,n}\geq0$ both converging to 0, part (4) of Theorem \ref{theo_df} provides at the limit a renormalized solution to \eqref{eq_system}.

\section{From weak solutions to renormalized solutions in the presence of drag forces}
This section is dedicated to the proof of part (3) of Theorem \ref{theo_df}. In the whole section we will assume that 
$\kappa>0$, $r_0>0$, and $r_1>0$, and we will consider a fixed weak solution $(\sro, \sro u)$ as in Definition \eqref{def_weak}. Let us define $\overline{g}_\varepsilon$ for any function $g$ as
$$
\overline{g}_\varepsilon(t,x)=\eta_\varepsilon\ast g(t,x), \qquad t>\eps,
$$
where 
$$
\eta_\eps(t,x)=\frac{1}{\eps^{d+1}}\eta_1(t/\eps,x/\eps),
$$
with $\eta_1$ a smooth nonnegative even function compactly supported in the space time ball of radius 1, and with integral equal to 1. 

\vskip0.3cm
Formally, we can show that a weak solution is also a renormalized solution by multiplying the equation by $\vfi'(u)$. However, solutions of \eqref{eq_system_r} have a limited amount of regularity. This has to be performed carefully. Let us explain the difficulties. First let us focus on the term
$$
\Dv (\sqrt{\nu\rho}\S).
$$
One way to obtain the renormalized equation from the weak one, is to consider the family of test functions $\psi\vfi'(\overline{u}_\varepsilon)$. We need to pass to the limit in the expression
$$
\int_0^\infty\int_\O \psi\sqrt{\nu\rho}\S \nabla \vfi'(\overline{u}_\eps)\,dx\,dt.
$$
But we cannot pass into the limit for the term $\sqrt{\nu\rho} \nabla \vfi'(\overline{u}_\eps)=\sqrt{\nu\rho} \vfi''(\overline{u}_\eps)\nabla \overline{u}_\eps$. Note that this term is different from $\overline{\T}_\eps$. The problem is that $\nabla u$ is not bounded in any functional space. 
\vskip0.3cm
An other difficulty is to obtain, in  the sense of distribution, the equality
$$
\vfi'(u)\partial_t(\rho u)= \partial_t(\rho\vfi(u))+(\vfi'(u)u-\vfi(u))\partial_t\rho.
$$
Indeed, we have absolutely no estimate available on $\partial_t u$.  Following Di Perna and Lions, \cite{Lions1996}, this can be obtained using commutators estimates which requires more a priori estimates than can be formally intuited. 
\vskip0.3cm
To solve these problems, we  need to introduce a cut-off function in $\rho$, $\phi_m(\rho)$ where $\phi_m$ is defined for every $m>0$ as 
\begin{equation}\label{eq_phirho}
\phi_m(y)=\left\{
 \begin{array}{ll}
0,& \mathrm{for} \ \ 0\leq y\leq \dfrac{1}{2m},\\[0.3cm]
2 m y-1,&  \mathrm{for} \ \ \dfrac{1}{2m}\leq y\leq \dfrac{1}{m},\\[0.3cm]
1, & \mathrm{for} \ \ \dfrac{1}{m}\leq y\leq m,\\[0.3cm]
2-y/m, & \mathrm{for} \ \ m\leq y\leq 2m,\\[0.3cm]
0,  &\mathrm{for} \ \ y\geq 2m.
 \end{array}
 \right.
 \end{equation}
 We will now work on 
 \begin{equation}\label{def_vm}
 v_m=\phi_m(\rho)u
 \end{equation}
  instead of $u$. Note that 
$$
 \nabla v_m =\frac{\phi_m(\rho)}{\sqrt{\nu\rho}}\T+4\sro\phi'_m(\rho)\srof u \nabla \srof.
 $$
 For $m$ fixed, $\frac{\phi_m(\rho)}{\sqrt{\nu\rho}}$ and $\sro\phi'_m(\rho)$ are bounded, and so $ \nabla v_m$ is bounded in $L^2$, thanks to the a priori estimates obtained from $\kappa>0$ and $r_1>0$ in \eqref{eq_a_priori}.
 \vskip0.3cm
 Obtaining the equation on $v_m$ is pretty standard, thanks to the extra regularity on the density provided by the quantum term, and the BD entropy. However, to highlight where are the difficulties, we will provide a complete proof. 
 
 \subsection{Preliminary lemmas}
 In this subsection, we introduce two standard lemmas to clarify the issues. The second one use the commutator estimates of Di Perna and Lions \cite{Lions1996}.
 \begin{lemm}\label{lemm_trivial}
 Let $g\in L^p(\R^+\times\O)$ and $h\in L^q(\R^+\times\O)$ with $1/p+1/q=1$ and $H\in W^{1,\infty}(\R)$. We denote by $\partial$ a partial derivative with respect to one of the dimension (time or space). Then we have:
 \begin{eqnarray*}
 &&\int_0^\infty\int_\O \overline{g}_\eps h \,dx\,dt=\int_0^\infty\int_\O g \overline{h}_\eps \,dx\,dt,\\
 &&\lim_{\eps\to0}\int_0^\infty\int_\O \overline{g}_\eps h \,dx\,dt=\int_0^\infty\int_\O g h \,dx\,dt,\\
 &&\partial \overline{g}_\eps=\overline{[\partial g]}_\eps,\\
&& \lim_{\eps\to0}\|H(\overline{g}_\eps)-H(g)\|_{L^s_{\mathrm{loc}}(\R^+\times\O)}=0,\qquad \mathrm{for \ any \ } 1\leq s<\infty,\\
 &&\partial H(g)=H'(g)\partial g\in L^r(\R^+\times\O) \qquad \mathrm{as \ long \  as \  } \ \partial g\in L^r(\R^+\times\O).
  \end{eqnarray*}
  
 \end{lemm}
 This lemma is very standard and then we omit the proof. We use also in the sequel the following lemma due to Lions (see \cite{Lions1996}).
 \begin{lemm}\label{lemm_DPLions}
 Let $\partial$ be a partial derivative in one direction (space or time).
 Let $g, \partial g\in L^p(\R^+\times\O),\,h\in L^{q}(\R^+\times\O)$ with $1\leq p,q\leq \infty$, and $\frac{1}{p}+\frac{1}{q}\leq 1$. Then, we have
 $$\|[\overline{\partial(gh)]}_\eps-\partial(g\overline{h}_\eps)\|_{L^{r}(\R^+\times\O)}\leq C\|\partial g\|_{L^{p}(\R^+\times\O)}\|h\|_{L^{q}(\R^+\times\O)}$$
 for some constant $C\geq 0$ independent of $\varepsilon$, $g$ and $h$, and with $r$ given by $\frac{1}{r}=\frac{1}{p}+\frac{1}{q}.$ In addition,
  $$\overline{[\partial(gh)]}_\eps-\partial(g\overline{h}_\eps)\to0\;\;\text{ in }\,L^{r}(\R^+\times\O)$$
 as $\varepsilon \to 0$ if $r<\infty.$
\end{lemm}

 \subsection{Equation on $v_m$}
 
 This subsection is dedicated to show that for any $\psi\in \Cin$, and any $m>0$, we have 
\begin{equation}\label{eq_phi_m}
\int_0^\infty\int_\O\left\{ \partial_t\psi\phi_m(\rho)-\psi \left(u\cdot\nabla\phi_m(\rho)+\phi'_m(\rho)\frac{\sro}{\sqrt\nu}\mathrm{Tr}(\T)\right)\right\} \,dx\,dt=0.
\end{equation}

\begin{equation}\label{eq_vm}
 \begin{split}
 &\int_0^\infty\int_\O\left\{ \partial_t\psi\rho v_m+\nabla\psi\cdot\left(\rho u\otimes v_m  -\phi_m(\rho)\sro\left({2}\sqrt{\nu}\S+{2}\sqrt{\kappa}\Sk+{\sqrt{\kappa}}\M\right)\right)\right.\\
 &\qquad \left.+\psi\left(-\frac{\sro}{\sqrt\nu}\mathrm{tr}(\T)\phi'_m(\rho)\rho u -{\sqrt{\kappa\rho}}\M\nabla\phi_m(\rho)+\phi_m(\rho) F\right)\right\}\,dx\,dt=0,
 \end{split}
 \end{equation}
 where $F$ is defined in \eqref{eq_F}, $v_m$ in \eqref{def_vm},  $\Sk$ is in \eqref{eq_quantic}, and $\S$  is the symmetric part of $\T$ defined in\eqref{eq_viscous}.
 \vskip0.3cm
 We begin with a list of a priori estimates. Note that they depends on all the parameters $r_0>0$, $r_1>0$,  $\kappa>0$, and $m$.
 \begin{lemm}\label{lemm_estimates}
 There exists a constant $C>0$ depending only on  the fixed solution $(\sro,\sro u)$, and $C_m$ depending also on $m$ such that 
 \begin{eqnarray*}
 &&\|\rho\|_{L^5(\R^+\times\O)}+ \|\rho u\|_{L^{5/2}(\R^+\times\O)}+ \|\rho |u|^2+\sro(|\S|+|\Sk|+|\M|+|f|)\|_{L^{5/3}(\R^+\times\O)}\\
 &&\qquad\qquad\qquad+ \|\rho^{\gamma/2}\nabla\rho^{\gamma/2} \|_{L^{5/4}(\R^+\times\O)}+\|r_0 u\|_{L^2(\R^+\times\O)}+\|r_1\rho|u|^2u\|_{L^{5/4}(\R^+\times\O)}\leq C\\
 && \|\nabla\phi_m(\rho)\|_{L^4(\R^+\times\O)}+ \|\partial_t\phi_m(\rho)\|_{L^2(\R^+\times\O)}\leq  C_m.
 \end{eqnarray*}
 \end{lemm}
 \begin{proof}
 From \eqref{eq_a_priori}, {$\sro \in L^\infty(\R^+,L^2(\O))$}, $\nabla\sro\in L^\infty(\R^+,L^2(\O))$ and $\nabla^2\sro\in L^2(\R^+\times\O)$. Hence, using 
 {Gagliardo-Nirenberg inequality}, $\sro\in L^\infty(\R^+;L^6(\O))$, and $\nabla\sro\in L^2(\R^+;L^6(\O))$. Since $\nabla\rho=2\sro\nabla\sro$, $\nabla\rho\in L^2(\R^+,L^3(\O))$. And 
 $\rho^2 \in L^\infty(\R^+;L^{3/2}(\O))$, so $\nabla(\rho^2)=2\rho\nabla\rho\in  L^2(\R^+;L^{3/2}(\O))$. This gives 
 $$
 \rho^2\in L^2(\R^+;L^{3}(\O))\cup L^\infty(\R^+;L^{3/2}(\O)).
 $$ 
 By interpolation, $\rho^2$ lies in all the space $L^p(L^q)$ with
 $$
 \frac{\alpha}{2}=\frac{1}{p},\qquad \frac{\alpha}{3}+\frac{2(1-\alpha)}{3}=\frac{1}{q},
 $$
 for $0\leq \alpha\leq 1$. For $\alpha= 4/5$, we obtain $\rho^2\in L^{5/2}(\R^+\times\O)$.
 We have
 $$
 \rho u=\rho^{3/4} \srof u. 
 $$
 From the $r_1$ term of $\DEr$, $\srof u\in L^4(\R^+\times\O)$, and we have $\rho^{3/4}\in L^{20/3}(\R^+\times\O)$. Hence $\rho u\in L^{5/2}(\R^+\times\O)$.
 The term $|\S|+|\Sk|+|\M|+|f|+\sro |u|^2\in L^2(\R^+\times\O)$ and $\sro\in L^{10}(\R^+\times\O)$, so $\rho |u|^2+\sro(|\S|+|\Sk|+|\M|+|f|)\in  L^{5/3}(\R^+\times\O)$.
From the a priori estimates, we have $\nabla\rho^{\gamma/2}\in  L^2(\R^+\times\O)$, and $\rho^{\gamma/2}\in L^\infty(\R^+,L^2(\O))$. Using 
{Galgliardo-Nirenberg inequality} we have $\rho^{\gamma/2}\in L^{10/3}(\R^+\times\O)$. So $\rho^{\gamma/2}\nabla\rho^{\gamma/2} \in L^{5/4}(\R^+\times\O)$.

\vskip0.3cm
{The estimate on $r_0u$ comes directly from the a priori estimates. Since $\sqrt{\rho}|u|^2$ is in $L^2(\R^+,L^2(\O))$ and $\rho^{1/4}u$ is in $L^4(\R^+ \times \O)$, we have $\rho^{1/4}u\sqrt{\rho}|u|^2 \in L^{4/3}(\R^+\times \O)$. Then using $\rho^{1/4} \in L^{20}(\R^+ \times \O)$ we obtain the result for $r_1\rho|u|^2u$.}

\vskip0.3cm
From Lemma \ref{lemm_trivial}, $\nabla\phi_m(\rho)=4\rho^{3/4}\phi'_m(\rho)\nabla\srof$. But $y\to 4y^{3/4}\phi'_m(y)$ is  bounded, and {by \eqref{eq_a_priori}} $\nabla\srof\in L^4(\R^+\times\O)$, so $\nabla\phi_m(\rho)\in  L^4(\R^+\times\O)$. From Lemma \ref{lemm_trivial}, and \eqref{eq_viscous},
\begin{eqnarray*}
&&\partial_t \phi_m(\rho)=\phi'_m(\rho)\partial_t\rho=-\phi'_m(\rho)\Dv(\rho u)=-\phi'_m(\rho)(\sqrt{\rho/\nu}\mathrm{Tr}(\T)+2\sro u\cdot\nabla\sro)\\
&&\qquad\qquad\qquad\qquad =-\sro\phi'_m(\rho) (\sqrt{1/\nu}\mathrm{Tr}(\T)+4\srof u\cdot\nabla\srof).
\end{eqnarray*}
Hence $\partial_t\phi_m(\rho)\in L^2(\R^+\times\O)$.

 \end{proof}
We use the first equation in Definition  \ref{def_weak} with test function $\overline{[\phi'_m(\overline{\rho}_\eps)\psi]}_\eps$.From Lemma \ref{lemm_trivial}, and \eqref{eq_viscous} we get the following. 
\begin{eqnarray*}
&&0=\int_0^\infty\int_\O\left\{ \partial_t\overline{\psi\phi'_m(\overline{\rho}_\eps)}_\eps \rho +\rho u\cdot\nabla\overline{\psi\phi'_m(\overline{\rho}_\eps)}_\eps\right\} \,dx\,dt\\
&&=-\int_0^\infty\int_\O\left\{ \psi\phi'_m(\overline{\rho}_\eps) \partial_t\overline{\rho}_\eps +\Dv(\overline{\rho u}_\eps)\psi\phi'_m(\overline{\rho}_\eps)\right\} \,dx\,dt\\
&&=\int_0^\infty\int_\O\left\{ \partial_t\psi\phi_m(\overline{\rho}_\eps)  -\psi\phi'_m(\overline{\rho}_\eps) \left[\overline{\frac{\sro}{\sqrt\nu}\mathrm{Tr}(\T)}_\eps+2 \overline{\sro u \cdot \nabla\sro}_\eps\right]\right\} \,dx\,dt.
\end{eqnarray*}
 We have  $\rho \in L^5(\R^+\times\O)$ and $\partial_t  [\phi'_m(\rho)\psi]\in L^{5/4}(\R^+\times\O)$, since $\partial_t\rho\in L^2(\R^+\times\O)$ and $\psi$ is $C^1$ compactly supported.  From Lemma \ref{lemm_estimates}, $\rho u\in L^{5/2}(\R^+\times\O)$, and $\nabla[\phi_m'(\rho)\psi]\in  L^{5/3}(\R^+\times\O)$, since $\psi$ is regular and compactly supported, and $\nabla\phi_m'(\rho)\in L^{4}(\R^+\times\O)$.  So we can pass to the limit $\eps\to0$ using Lemma \ref{lemm_trivial}.
We obtain:
$$
0=\int_0^\infty\int_\O\left\{ \partial_t\psi\phi_m(\rho)  -\psi\phi'_m(\rho) \left[\frac{\sro}{\sqrt\nu}\mathrm{Tr}(\T)+2 \sro u\cdot\nabla\sro\right]\right\} \,dx\,dt.
$$ 
Since $\psi\nabla\phi_m(\rho)$ lies in $L^4(\R^+\times\O)$ and is compactly supported, and $u$ is in $L^2(\R^+\times\O)$, using the Lemma \ref{lemm_trivial} we find
$$
0=\int_0^\infty\int_\O\left\{ \partial_t\psi\phi_m(\rho)  -\psi\left[\phi'_m(\rho) \frac{\sro}{\sqrt\nu}\mathrm{Tr}(\T)+ u\cdot\nabla\phi_m(\rho)\right]\right\} \,dx\,dt,
$$ 
{which is \eqref{eq_phi_m}.}
\vskip0.3cm
 
 For $\eps$ small enough, we consider the  function $\overline{\psi\phi_m(\rho)}_\eps$ as test function in the Definition \ref{def_weak}. In the same way than above, from Lemma \ref{lemm_estimates} and Lemma \ref{lemm_trivial}, passing into the limit in $\eps$, we get  
\begin{eqnarray*}
 &&\int_0^\infty\int_\O\left\{ \partial_t\psi\rho v_m+\nabla\psi\cdot\left(\rho u\otimes v_m  -\phi_m(\rho)\sro\left({2}\sqrt{\nu}\S+{2}\sqrt{\kappa}\Sk+{\sqrt{\kappa}}\M\right)\right)\right.\\
 &&\qquad \left.+\psi \left(\partial_t \phi_m(\rho)+u\cdot\nabla\phi_m(\rho) \right)\rho u
+\psi\left( -{\sqrt{\kappa\rho}}\M\nabla\phi_m(\rho)+\phi_m(\rho) F\right)\right\}\,dx\,dt=0,
 \end{eqnarray*}
 Using \eqref{eq_phi_m} {this gives \eqref{eq_vm}}.

 \subsection{Equation of renormalized solutions}
 
 We use the function $\overline{\psi\vfi'(\overline{v_m}_\eps)}_\eps$ as a test function in \eqref{eq_vm}. Using  Lemma \ref{lemm_trivial}, we find:
 \begin{eqnarray*}
 &&\qquad\int_0^\infty\int_\O  \left(\partial_t\left[\overline{\psi\vfi'(\overline{v_m}_\eps)}_\eps\right]\rho v_m+\nabla\left[\overline{\psi\vfi'(\overline{v_m}_\eps)}_\eps\right]:(\rho u\otimes v_m )\right)\,dx\,dt\\
 &&=-\int_0^\infty\int_\O \psi\vfi'(\overline{v_m}_\eps)\left(\partial_t\overline{[\rho v_m]}_\eps+\Dv\overline{(\rho u\otimes v_m )}_\eps\right)\,dx\,dt
 \end{eqnarray*}
Thanks to Lemma \ref{lemm_estimates}, we can use Lemma \ref{lemm_DPLions}, with $g=\rho$ and $h=v_m$, and then $g=\rho u$ and $h=v_m$.  Note that $v_m \in L^4(\R^+\times\O)$. So, the expression above as the same limit when $\eps$ goes to zero than
\begin{eqnarray*}
 &&-\int_0^\infty\int_\O \psi\vfi'(\overline{v_m}_\eps)\left(\partial_t\left(\rho\overline{v_m}_\eps\right)+\Dv\left(\rho u\overline{ v_m }_\eps\right)\right)\,dx\,dt.
 \end{eqnarray*}
Thanks to the first equation in Definition \ref{def_weak} this is equal to
\begin{eqnarray*}
 &&\qquad-\int_0^\infty\int_\O \psi\vfi'(\overline{v_m}_\eps)\left(\rho\partial_t\overline{v_m}_\eps+\rho u\cdot\nabla\overline{ v_m }_\eps\right)\,dx\,dt\\
 &&=-\int_0^\infty\int_\O \psi\left(\rho\partial_t\vfi(\overline{v_m}_\eps)+\rho u\cdot\nabla\vfi(\overline{v_m}_\eps)\right)\,dx\,dt\\
 &&=\int_0^\infty\int_\O \vfi(\overline{v_m}_\eps)\left(\rho\partial_t\psi+\rho u\cdot\nabla\psi)\right)\,dx\,dt,
 \end{eqnarray*}
 which converges, when $\eps$ goes to 0, to 
\begin{equation}\label{vianet1}
 \int_0^\infty\int_\O \vfi(v_m)\left(\rho\partial_t\psi+\rho u\cdot\nabla\psi)\right)\,dx\,dt.
 \end{equation}
 
 {Note that  
\begin{eqnarray*}
\nabla v_m&=&\nabla\left(\frac{\phi_m(\rho)}{\rho}\rho u\right)\\
&=&\nabla\left(\frac{\phi_m(\rho)}{\rho}\right)\rho u+\frac{\phi_m(\rho)}{\rho}\nabla(\rho u)\\
&=&\nabla\left(\frac{\phi_m(\rho)}{\rho}\right)\rho u+\frac{\phi_m(\rho)}{\sqrt{\nu}\sro}\T+2\frac{\phi_m(\rho)}{\rho}\sro u\cdot\nabla\sro\\
&=&4\sro \phi'_m(\rho) \rho^{1/4}u\cdot\nabla\rho^{1/4}.
\end{eqnarray*} 
  Note that $\sro\phi_m'(\rho)$  is bounded.
 So, thanks to the third line of \eqref{eq_a_priori}, we have $\nabla v_m\in L^2(\R^+\times\O)$, and so $\nabla \vfi''(v_m)\in L^2(\R^+\times\O)$. 
 }

 So, thanks to Lemma \ref{lemm_trivial} we can pass into the limit in the other terms and find 
 \begin{eqnarray}
&&\qquad\int_0^\infty\int_\O\left\{ -\nabla\left(\psi\vfi'(v_m)\right))\phi_m(\rho)\sro\left({2}\sqrt{\nu}\S+{2}\sqrt{\kappa}\Sk+{\sqrt{\kappa}}\M\right)\right. \nonumber \\
 && \left.+\psi\vfi'(v_m)\left(-\frac{\sro}{\sqrt\nu}\mathrm{Tr}(\T)\phi'_m(\rho)\rho u -{\sqrt{\kappa\rho}}\M\nabla\phi_m(\rho)+\phi_m(\rho) F\right)\right\}\,dx\,dt. \label{vianet2}
 \end{eqnarray}
Putting \eqref{vianet1} and \eqref{vianet2} together gives
 
 \begin{eqnarray*}
 &&\int_0^\infty\int_\O \left\{\vfi(v_m)\left(\rho\partial_t\psi+\rho u\cdot\nabla\psi)\right)\right.\\
 &&\qquad -\nabla\psi\vfi'(v_m)\phi_m(\rho)\sro\left({2}\sqrt{\nu}\S+{2}\sqrt{\kappa}\Sk+{\sqrt{\kappa}}\M\right)\\
 &&\qquad -\psi \vfi''(v_m)\nabla v_m\phi_m(\rho)\sro\left({2}\sqrt{\nu}\S+{2}\sqrt{\kappa}\Sk+{\sqrt{\kappa}}\M\right)\\
 &&\qquad  \left.+\psi\vfi'(v_m)\left(-\frac{\sro}{\sqrt\nu}\mathrm{Tr}(\T)\phi'_m(\rho)\rho u -{\sqrt{\kappa\rho}}\M\nabla\phi_m(\rho)+\phi_m(\rho) F\right)\right\}\,dx\,dt. 
 \end{eqnarray*}
 We now pass into the limit $m$ goes to infinity.
 From the a priori estimates \eqref{eq_a_priori}, $r_0 \ln\rho$ lies in $L^\infty(\R^+;L^1(\O))$, so $\rho>0$ almost everywhere, and 
 \begin{eqnarray*}
 &&\phi_m(\rho) \mathrm{ \ \ converges \ to \ 1,} \qquad \mathrm{for \ almost \ every \ }(t,x)\in \R^+\times\O, \\
&&v_m \mathrm{ \ \ converges \ to \ } u, \qquad \mathrm{for \ almost \ every \ }(t,x)\in \R^+\times\O, \\
&& | \rho \phi'_m(\rho) | \leq 2, \ \qquad \mathrm{and \ converges \ to \ }0  \qquad \mathrm{for \  almost \ every } \ (t,x)\in \R^+\times\O.
 \end{eqnarray*}
 Now, using that $\phi_m$ is compactly supported in $\R^+$, we get
 \begin{eqnarray*}
 \sro \nabla v_m&=&\frac{\phi_m(\rho)}{\sro}\rho\nabla u+4\srof u \rho\phi_m'(\rho)\nabla\srof\\
 &=& \frac{\phi_m(\rho)}{\sro}\left[\nabla(\rho u)-2\sro{u \cdot}\nabla\sro\right]+4\srof u \rho\phi_m'(\rho)\nabla\srof\\
 &=&\phi_m(\rho) \frac{\T}{\sqrt\nu}+4\srof u \rho\phi_m'(\rho)\nabla\srof,
 \end{eqnarray*}
 which, thanks to the dominated convergence theorem, converges to $\T / \sqrt{\nu}$ in $L^2(\R^+\times\O)$. Hence, passing into the limit $m\to \infty$, we find
 \begin{eqnarray*}
 &&\int_0^\infty\int_\O \left\{\vfi(u)\left(\rho\partial_t\psi+\rho u\cdot\nabla\psi)\right) -\nabla\psi\vfi'(u)\sro\left({2}\sqrt{\nu}\S+{2}\sqrt{\kappa}\Sk+{\sqrt{\kappa}}\M\right)\right. \\
&&  \left.-\psi \vfi''(u)\frac{\T}{\sqrt{\nu}}\left({2}\sqrt{\nu}\S+{2}\sqrt{\kappa}\Sk+{\sqrt{\kappa}}\M\right) +\psi\vfi'(u)F\right\}\,dx\,dt. 
 \end{eqnarray*}
 This gives the second equation in the definition of renormalized solutions with 
 $$
 R_\vfi= \vfi''(u)\frac{\T}{\sqrt{\nu}}\left({2}\sqrt{\nu}\S+{2}\sqrt{\kappa}\Sk+{\sqrt{\kappa}}\M\right).
 $$
 \vskip0.3cm
 We want now to show \eqref{eq_viscous_renormalise}. As above, multiplying \eqref{eq_viscous} by $\overline{[\phi_m(\rho)\vfi_i'(u)]}_\eps$, and passing into the limit $\eps$ goes to 0, we find
 \begin{eqnarray*}
 &&\phi_m(\rho)\vfi'_i(u)\sqrt{\nu\rho}\T={\nu}\nabla(\phi_m(\rho)\vfi'_i(u)\rho u)-4\rho \phi'_m(\rho)\nabla\srof \srof u \sro \vfi_i'(u)\\
 &&-\vfi_i''(u)\frac{\T}{\sqrt{\nu}}\sro u\phi_m(\rho)-2{\nu}\phi_m(\rho)\vfi'_i(u)\sro{u\cdot}\nabla\sro.
 \end{eqnarray*}
 Passing into the limit $m$ goes infinity, we recover  \eqref{eq_viscous_renormalise}, with 
 $$
 \overline{R}_\vfi= -\vfi_i''(u)\frac{\T}{\sqrt{\nu}}\sro u.
 $$
 Hence, $(\sro,\sro u)$ is a renormalized solution.
\section{From renormalized solutions to weak solutions in the general case}
This section is dedicated to the proof of (2) in Theoren \ref{theo_df}. We consider $(\sro,\sro u)$, a renormalized solution as defined in Definition \ref{def_renormalise}, in the general case where $r_0\geq0$, $r_1\geq0$ and $\kappa\geq0$, but $\nu>0$. We want to show that it is also a weak solution as defined in Definition \ref{def_weak}. Let $\Phi:\R\to\R$ be a nonnegative smooth function compactly supported, equal to 1 on $[-1,1]$, and $\tilde{\Phi}(z)=\int_0^z\Phi(s)\,ds$. Then we define for $y\in \R^d$
$$
\vfin(y)=n\tilde{\Phi}(y_1/n)\Phi(y_2/n)\cdot\cdot\cdot\Phi(y_d/n).
$$
Note that $\vfin$ lies in $W^{2,\infty}(\R^d)$ for any fixed $n$, $\vfin$ converges everywhere to $y\to y_1$, $\vfin'$ is uniformly bounded in $n$ and converges everywhere to $(1,0,\cdot\cdot\cdot,0)$, and 
$\|\vfin''\|_{L^\infty(\R)}\leq C/n$ converges to 0, when $n$ converges to infinity. Hence $R_{\vfin}$ and $\overline{R}_\vfin$ both converge  to 0 in the sense of measure when $n$ converges to infinity.
We use this function $\vfin$ in the second equation of the Definition \ref{def_renormalise}. 
Using the Lebesgue's Theorem for the limit $n\to \infty$, we get the equation on $\rho u_1$ in the Definition \ref{def_weak}. permuting the directions, we get the full vector equation on $\rho u$ in Definition \ref{def_weak}.  

We use again the Lebesgue's dominated convergence Theorem to pass into the limit in \eqref{eq_viscous_renormalise} with $i=1$ and the function $\vfin$ to obtain \eqref{eq_viscous}. Hence, the renormalized solution is also a weak solution.

\section{{Stability and existence of weak renormalized solutions}}

This section is dedicated to the proof of (4) {and (1)} in Theorem \ref{theo_df}. We consider sequences $r_{0,n}, r_{1,n}, \kappa_n, \nu_n, \rho_n, u_n$ as in the hypothesis of Theorem \ref{theo_df}.  We begin to show the following lemma:
\begin{lemm}\label{lemm_stabilite}
{Up to} a subsequence, still denoted $n$, {the following properties hold.}
\begin{enumerate}
\item The sequence $\rho_n$ converges strongly to $\rho$ in $C^0(\R^+;L^p_\mathrm{loc}(\O))$ for $1\leq p<\sup(3,\gamma)$.
\item The sequence $\rho_n u_n$ converges to $\rho u$ in $C^0(\R^+; L^{3/2}(\O)-weak)$, and strongly in $L^p_{\mathrm{loc}}(\R^+;L^q(\O))$ for $1\leq p<\infty$, and $1\leq q<3/2$. 
\item The sequences ${\T}_{,n}, {\S}_{,n}, {{\Sk}_{,n}} $ converge weakly in $L^2(\R^+\times\O)$ to $\T, \S, \Sk$. 
\item For every function $H \in W^{2,\infty}(\R^d)$, and $0<\alpha< 5\gamma/3$, we have that   $\rho^\alpha_nH(u_n)$ converges
 strongly in $L_{\mathrm{loc}}^p(\R^+\times\O)$ to $\rho^\alpha H(u)$ for $1\leq p< 5\gamma/(3\alpha)$.
\item If $\lim_{n\to\infty}r_{1,n}=r_1>0$, then $\rho^{1/3}u_n$ converges to $\rho^{1/3} u$  in $L^p_{\mathrm{loc}}(\R^+;L^q(\O))$ for $1\leq p<4$, and $1\leq q< 18/5$.
\item If $\lim_{n\to\infty}r_{0,n}=r_0>0$, then $r^{1/2}_{0,n}u_n$ converges to $r^{1/2}_{0} u$  in $L^p_{\mathrm{loc}}(\R^+\times\O)$ for $1\leq p<2$.
\item Consider a smooth and  increasing function $h:\R^+\to \R^+$ such that $h(y)=y^{3/4}$ for $y<1$ and $h(y)=y^{1/2}$ for $y>2$.
If  $\lim_{n\to\infty}\kappa_{n}=\kappa>0$, then $\nabla h(\rho_n)$ converges to $\nabla h(\rho)$ in $L^2_{\mathrm{loc}}(\R^+; L^p(\O))$, for $1\leq p<6$.
\end{enumerate}
\end{lemm}
\begin{proof}
From (1) to (4), we use only a priori estimates which are non dependent on $r_0$, $r_1$, and $\kappa$. Then $\rho_n u_n$ is uniformly bounded in $L^\infty(\R^+;L^{3/2}(\O))$. From the continuity equation $\partial_t \rho_n$ is uniformly bounded in $L^\infty(\R^+;W^{-1,3/2}(\O))$. Moreover $\nabla\rho_n$ is uniformly bounded in $L^\infty(\R^+;L^{3/2}(\O))$, $\rho_n$ is uniformly bounded in $L^\infty(\R^+;L^3(\O)\cap L^{\gamma}(\O))$. Hence, using  Aubin Simon's  Lemma  $\rho_n$ is compact in $C^0(\R^+;L^p_\mathrm{loc}(\O))$ for $1\leq p<\sup(3,\gamma)$. 
\vskip0.3cm
From the second equation in the Definition \ref{def_weak},  the sequence $\partial_t(\rho_n u_n)$ is uniformly bounded in $L^2(\R^+; H^{-N}(\O)) $ for a $N$ big enough. From \eqref{eq_viscous}, $\nabla(\rho_n u_n)$ is uniformly bounded in $L^2_\mathrm{loc}(\R^+;L^{3/2}(\O))$. Together with $\rho_n u_n$ uniformly bounded in $L^\infty(\R^+;L^{3/2}(\O))$ this gives the strong compactness of $\rho_n u_n$ in $L^p_{\mathrm{loc}}(\R^+;L^q(\O))$ for $1\leq p<\infty$, and $1\leq q<3/2$. 
\vskip0.3cm
The sequences ${\T}_{,n}, {\S}_{,n}, {{\Sk}_{,n}} $ are uniformly bounded in $L^2(\R^+\times\O)$, and so, up to a subsequence, converge weakly in $L^2(\R^+\times\O)$ to  functions $\T, \S, \Sk$. 
\vskip0.3cm
From (1) and (2), up to a subsequence still denoted $n$, $\rho_n$ and $\rho_n u_n$ converge almost everywhere respectively to $\rho$ and $\rho u$.
So, for almost every $(t,x)$ such that $\rho(t,x)>0$, $u_n(t,x)$ converges to $u(t,x)$, and so $\rho_n(t,x)^\alpha H(u_n(t,x))$ converges to $\rho(t,x)^\alpha H(u(t,x))$. But for almost every $(t,x)$ such that $\rho(t,x)=0$, $|\rho_n(t,x)^\alpha H(u_n(t,x))|\leq C \rho_n(t,x)^\alpha$ which converges to $0=\rho(t,x)^\alpha H(u(t,x))$ since $\alpha>0$. So, we have convergence almost everywhere. And  $\rho^\alpha_nH(u_n)$  is uniformly bounded in $L^{5\gamma/(3\alpha)}(\R^+\times\O)$. {Indeed $\rho^{\gamma/2}\in L^\infty(\R^+,L^2(\O))\cap L^2(\R^+,L^6(\O))$,and  by interpolation,
$\rho^{\gamma/2}\in L^{10/3}(L^{10/3})$ .
} Hence, we have strong convergence in $L_{\mathrm{loc}}^p(\R^+\times\O)$ to $\rho^\alpha H(u)$ for $1\leq p< 5\gamma/(3\alpha)$. 
\vskip0.3cm
We assume that $\lim_{n\to\infty}r_{1,n}=r_1>0$. 
{ We have
$\rho_n^{1/3}u_n=\rho_n^{1/4}u_n \rho_n^{1/12}$. Since $\sro_n$ is uniformly bounded in $L^\infty(\R^+,L^6(\O))$ (From Sobolev, since $\nabla \sro_n$ is uniformly bounded in $L^\infty(\R^+,L^2(\O))$), we have $\rho_n^{1/12}$ uniformly bounded in $L^\infty(\R^+,L^{36}(\O))$. Moreover, $1/4+1/36=5/18$.
}
Then 
 the functions $\rho_n^{1/3} u_n $ are uniformly bounded in  $L^4(\R^+;L^{18/5}(\O))$. 
We denote $\mathbf{1}_{\{\rho>0\}}$ the function which is equal to one on 
 $\{t,x | \rho(t,x)>0\}$ and zero on $\{t,x | \rho(t,x)=0\}$. The function $\mathbf{1}_{\{\rho>0\}}\rho_n^{1/3} u_n$ converges almost everywhere to $\mathbf{1}_{\{\rho>0\}}\rho^{1/3} u$ so the convergence holds in $L^p_{\mathrm{loc}}(\R^+;L^q(\O))$ for $1\leq p<\infty$, and $1\leq q< {18/5}$. {Note that for almost every $(t,x)$ such that $\rho(t,x)=0$, we have 
 $\rho_n(t,x)^{1/12}=(\rho_n(t,x)-\rho(t,x))^{1/12}$.} So,  for every $1\leq p<\infty$, and 
 $1\leq q <36$:
 \begin{eqnarray*}
&& \|\mathbf{1}_{\{\rho=0\}}\rho_n^{1/3} u_n\|_{L^{p_1}_{\mathrm{loc}}(\R^+;L^{q_1}(\O))}\leq \|\mathbf{1}_{\{\rho=0\}}\rho_n^{1/12}\|_{L^p_{\mathrm{loc}}(\R^+;L^q(\O))}\|\rho_n^{1/4}u_n\|_{L^4(\R^+\times\O)}\\
&&\qquad\qquad
\leq {C \|(\rho_n-\rho)^{1/12}\|_{L^{p}_{\mathrm{loc}}(\R^+;L^{q}(\O))}=}C \|\rho_n-\rho\|^{1/12}_{L^{p/12}_{\mathrm{loc}}(\R^+;L^{q/12}(\O))}
 \end{eqnarray*}
 converges to 0 when $n$ goes to {infinity}, where 
 $$
 \frac{1}{p_1}=\frac{1}{4}+\frac{1}{p}, \qquad \frac{1}{q_1}=\frac{1}{4}+\frac{1}{q}.
 $$
 So $\rho_n^{1/3} u_n {=\mathbf{1}_{\{\rho=0\} }\rho_n^{1/3} u_n+ \mathbf{1}_{\{\rho>0\}}\rho_n^{1/3} u_n}$ converges  to $\rho^{1/3} u$ in $L^p_{\mathrm{loc}}(\R^+;L^q(\O))$ for $1\leq p<\infty$, and $1\leq q< {18/5}$. 
\vskip0.3cm
Let us  assume that $\lim_{n\to\infty}r_{0,n}=r_0>0$. Then $\ln\rho_n$ is uniformly bounded in $L^\infty(\R^+;L^1(\O))$. The function $-\log$ is convex, so the limit $\rho$ verifies the same, and $\rho>0$ for almost every $(t,x)\in \R^+\times\O$. So $u_n$ converges almost everywhere to $u$, and $u_n$ is uniformly bounded in $L^2(\R^+\times\O)$. Hence $u_n$ converges to $u$ in $L^p_{\mathrm{loc}}(\R^+\times\O)$, for $1\leq p<2$.
\vskip0.3cm
Now we assume that  $\lim_{n\to\infty}\kappa_{n}=\kappa>0$. We have 
\begin{eqnarray*}
&&|\nabla (\nabla h(\rho_n))|=|2\sqrt{\rho_n}h'(\rho_n)\nabla^2\sqrt{\rho_n}+({8}\sqrt{\rho_n}h'(\rho_n)+16 h''(\rho_n)\rho_n\sqrt{\rho_n})(\nabla \rho_n^{1/4}\otimes\nabla\rho_n^{1/4})|\\
&&\qquad\qquad\qquad\leq C(|\nabla^2\sqrt{\rho_n}|+|\nabla \rho_n^{1/4}|^2).
\end{eqnarray*}
So $\nabla (\nabla h(\rho_n))$ is uniformly bounded in $L^2(\R^+\times\O)$. Moreover , using the continuity equation and \eqref{eq_viscous}, we get
\begin{eqnarray*}
&&|\partial_t h(\rho_n)|=|h'(\rho_n)\sqrt{\rho_n}\mathrm{Tr}\frac{{\T}_{,n}}{\sqrt{\nu}}+4\sqrt{\rho_n} u_n\cdot\nabla\rho_n^{1/4} h'(\rho_n)\rho_n^{1/4}|\\
&&\qquad\qquad \leq C(|{\T}_n|+\rho |u|^2+|\nabla\rho^{1/4}|^2),
\end{eqnarray*}
which is uniformly bounded in $L^2_{\mathrm{loc}}(\R^+,L^1(\O))$. Note that we cannot bound it in $L^2(\R^+\times\O)$ as in the previous section, since we cannot use that $r_{1,n}$ is bounded by below. We have shown that $\partial_{t}\nabla h(\rho_n)$ is uniformly bounded in $L^2(\R^+;W^{-2,6/5}(\O))$. Hence, using the Aubin Simon lemma, we find that $\nabla h(\rho_n)$ converges strongly to $\nabla h(\rho)$ in $L^2_{\mathrm{loc}}(\R^+; L^p(\O))$, for $1\leq p<6$.
\end{proof}

\noindent{\underline{Proof of part (4) of Theorem \ref{theo_df}}}
We are now ready to show the part (4) in Theorem \ref{theo_df}.
Using (1) and (2) of Lemma \ref{lemm_stabilite}, we can pass into the limit in the continuity equation. Using (1) (2) (3) and (4) of Lemma \ref{lemm_stabilite} we can pass into the limit into the first line of the second equation of Definition \ref{def_renormalise}. The sequence $R_{n,\vfi}$ is uniformly bounded in measures, so it converges to a measure $R_{\vfi}$ with the same bound. The function $f_n$ converges weakly in $L^2(\R^+ \times \O)$ to $f$ and, thanks to (4) of Lemma \ref{lemm_stabilite}, 
$\psi\sqrt{\rho_n} \vfi'(u_n)$ converges strongly in $L^2(\R^+ \times \O)$ to $\psi\sqrt{\rho} \vfi'(u)$, so we can pass into the limit in this term. $\nabla \rho_n^{\gamma/2}$ converges weakly in $L^2(\R^+ \times \O)$ to $\nabla\rho^{\gamma/2}$, and $\psi \rho_n^{\gamma/2}$ converges strongly to $\rho^{\gamma/2}$ thanks to (1) in Lemma \ref{lemm_stabilite}. So we can pass into the limit in the pressure term. If $r_{0,n}$ converges to 0, then $r_{0,n} u_n=r_{0,n}^{1/2}r_{0,n}^{1/2}u_n$ converges to 0 in $L^2(\R^+ \times \O)$, since $r_{0,n}^{1/2}u_n$ is uniformly bounded in $L^2(\R^+ \times \O)$. Otherwise, using (6) in Lemma \ref{lemm_stabilite}, it converges to $r_0 u$ in $L^1_{\mathrm{loc}}(\R^+\times\O)$. We can treat the term $r_{1,n}$ in the same way using (5) in Lemma \ref{lemm_stabilite}. So the two equations of Definition \ref{def_renormalise} are verified at the limit. Thanks to (1) and (2) of Lemma \ref{lemm_stabilite}, we can pass into the limit for the initial values. It remains to pass into the limit in \eqref{eq_viscous_renormalise}, and \eqref{eq_quantic}. The measures $\overline{R}_{n,\vfi}$ are uniformly bounded in measures, so they converge to a measure with the same bound. The functions $\nabla\sqrt{\rho_n}$ converge weakly to $\nabla\sro$ in $L^2_{\mathrm{loc}}(\R^+ \times \O)$. So, using (1) (3) and (4) of Lemma \ref{lemm_stabilite}, we can pass into the limit in \eqref{eq_viscous_renormalise}.  If $\kappa_n$ converges to 0, then $\sqrt{\kappa_n}\Sk$ converges to 0 weakly in $L^2(\R^+ \times \O)$. Otherwise, $\nabla^2\sqrt{\rho_n}$ converges weakly in $L^2(\R^+ \times \O)$ to $\nabla^2\sro$ and, thanks to (1) of   Lemma \ref{lemm_stabilite}, $\sqrt{\rho_n}\nabla^2\sqrt{\rho_n}$ converges weakly to $\sqrt{\rho}\nabla^2\sqrt{\rho}$. Note that 
$$
\sqrt{\rho_n} \nabla (\srof_n\otimes\srof_n)= g(\rho_n) \nabla (h(\rho_n)\otimes \srof_n),
$$
with $4 h'(\rho_n) \srof_n g(\rho_n)=1$. Especially, $|g(\rho_n)|\leq 1+\srof_n$. So, thanks to (1) of Lemma \ref{lemm_stabilite},  $g(\rho_n)$ converges strongly to $g(\rho)$ in $L^4(\R^+,L^2(\O))$, $\nabla\srof_n$ converges weakly  to $\nabla\srof$ in $L^4(\R^+,L^4(\O))$, and $\nabla h(\rho_n)$ converges strongly to $\nabla h(\rho)$ in $L^2(\R^+,L^4(\O))$ thanks to (7) in Lemma \ref{lemm_stabilite}. Hence we can pass into the limit in \eqref{eq_quantic}. This ends the proof of (4) in Theorem \ref{theo_df}.

\vskip0.3cm
\noindent{\underline{Proof of part (1) in Theorem \ref{theo_df}}}.
Consider sequences  $r_{0,n}>0$, $r_{1,n}>0$ and $\kappa_n>0$, converging respectively to $r_0\geq0$,   $r_1\geq0$ and  $\kappa\geq0$. For $n$ fixed, thanks to \cite{VasseurYu2015}, there exists a weak solution in the sense of Definition \ref{def_weak} to the system. Thanks to (2) in Theorem \ref{theo_df}, these solutions are  renormalized solutions in the sense of Definition \ref{def_renormalise}. Thanks to the stability result (4) in Theorem \ref{theo_df}, the limit is a renormalized solution for the system with coefficients $r_0$, $r_1$ and $\kappa$.  
\appendix
\section{A priori estimates}

The construction of weak solutions for $r_0>0, r_1>0$, and $\kappa>0$ has been done in \cite{VasseurYu2015} in the case $f=0$ and $\M=0$, using a Faedo-Galerkin method.
The construction can be straightforwardly extended to the case with source terms $f$ and $\M$, as long as the a priori estimates still hold. Note that, during the construction, the a priori estimates are proved to hold at the level of the Galerkin approximated solutions. To simplify the presentation, we will show  in this section, that the a priori estimates hold for any smooth solutions. This can be used to show existence of weak solutions as in \cite{VasseurYu2015} (see also \cite{GiLaVi2015} in the case of cold pressure). Note that we need the construction, and the a priori estimates only in the case where $r_0, r_1,\kappa$ are all positive. We actually proved that it is still valid in the case of some of these coefficients are 0 in (4) of Theorem \ref{theo_df}.
\vskip0.3cm
We consider a smooth solution $(\sro,\sro u)$ of (\ref{eq_system_r}). Multiplying the continuity equation by $\gamma \rho^{\gamma-1}/(\gamma-1)$ and the second  equation by $u$, integrating in $x$ the sum of these two quantities give
\begin{equation}\label{eq_appendix_energy}
\partial_t E(\sro,\sro u)+\De(\S)+r_0\|u\|^2_{L^2(\O)}+r_1 \int_\O \rho |u|^4\,dx=\int_\O \sro f u\,dx-\sqrt{\kappa}\int \M:\sro\nabla u\,dx.
\end{equation}
The equation on $v=u+\nu\nabla\ln\rho$ reads
$$
\partial_t (\rho v)+\Dv(\rho u\otimes v)+ \nabla\rho^\gamma -2\Dv (\nu\sro\A u+\sqrt{\kappa\rho}\Sk)=\sro f+\sqrt{\kappa}\Dv(\sro \M).
$$
Multiplying this equation by $v$, and adding the continuity equation multiplied by $\gamma \rho^{\gamma-1}/(\gamma-1)+\nu$ gives, after integrating in $x$:
\begin{equation}\label{eq_appendix_BD}
\begin{split}
&\partial_t \left\{ \mathcal{E}_{BD}(\rho,u)+r_0 \nu \int_\O(\rho-\ln\rho)\,dx\right\} +\mathcal{D}_{BD} (\rho,u)+r_0\|u\|^2_{L^2(\O)}+r_1 \int_\O \rho |u|^4\,dx\\
&\qquad =r_1\nu\int_\O 2\sro |u|^2u\cdot\nabla\sro\,dx+2\nu \int_\O f\nabla\sro\,dx-\sqrt{\kappa}\nu\int_\O \sro \nabla^2(\ln \rho):\M\,dx.
\end{split}
\end{equation}
We consider two cases.
\vskip0.3cm
\subsection{If $M$ is symmetric} Then, since 
$$
\M:\nabla u=\M:\D u,
$$
from \eqref{eq_appendix_energy}, we find
\begin{eqnarray*}
&&\qquad\qquad\partial_t E(\sro,\sro u)+\frac{1}{2}\De(\S)+r_0\|u\|^2_{L^2(\O)}+r_1 \int_\O \rho |u|^4\,dx\\
&&\leq \frac{\kappa}{\nu}\|\M(t)\|^2_{L^2(\O)}+\frac{\|f(t)\|_{L^2(\O)}}{\|f\|_{L^1(\R^+;L^2(\O))}} E(\sro,\sro u)+\|f(t)\|_{L^2(\O)}\|f\|_{L^1(\R^+;L^2(\O))}.
\end{eqnarray*}
The Gronwall's lemma gives that 
\begin{eqnarray*}
&&\sup_{t\in\R^+}E(\sro(t),\sro(t) u(t))+\frac{1}{2}\int_{\R^+}\De(\S)\,dt+r_0\|u\|^2_{L^2(\R^+\times\O)}+r_1 \int_0^\infty\int_\O \rho |u|^4\,dx\,dt\\
&&\qquad\qquad \leq 2\frac{\kappa}{\nu}\|\M\|^2_{L^2(\R^+\times\O)}+2\|f\|^2_{L^1(\R^+;L^2(\O))}+E(\sro_0,\sro_0 u_0).
\end{eqnarray*}
We have 
\begin{eqnarray*}
&&\qquad \qquad \left| r_1\nu\int_0^\infty\int_\O 2\sro |u|^2u\cdot\nabla\sro\,dx\,dt\right|\leq 3\nu r_1  \int_0^\infty\int_\O \sro|u|^2\sro|\nabla u|\,dx\,dt\\
&&\leq \nu\int_0^\infty\int_\O\rho|\nabla u|^2\,dx+{\dfrac{9}{4}}\nu r_1\int_0^\infty\int_\O r_1\rho |u|^4\,dx\,dt \\
&& \leq \nu \int_0^\infty\int_\O \rho|\A u|^2\,dx\,dt + (1+{\dfrac{9}{4}}\nu r_1) \left({2}\frac{\kappa}{\nu}\|\M\|^2_{L^2(\R^+\times\O)}+2\|f\|^2_{L^1(\R^+;L^2(\O))}+E(\sro_0,\sro_0 u_0)\right)\\
&& \leq \frac{1}{2}\int_0^\infty \mathcal{D}_{BD} (\rho,u)\,dx\,dt + (1+{\dfrac{9}{4}}\nu r_1) \left({2}\frac{\kappa}{\nu}\|\M\|^2_{L^2(\R^+\times\O)}+2\|f\|^2_{L^1(\R^+;L^2(\O))}+E(\sro_0,\sro_0 u_0)\right),
\end{eqnarray*}
and 
$$
\left| \nu \int_\O f\nabla\sro\,dx\right|\leq \sqrt{\nu} \|f\|_{L^2(\O)} \sqrt{\E(\sro(t),\sro(t) u(t))}
$$
and 
$$
\sqrt{\kappa}\nu\int_0^\infty\int_\O |\sro \nabla^2(\ln \rho):\M|\,dx\,dt \leq \frac{1}{4}\int_0^\infty \mathcal{D}_{BD} (\rho,u)\,dt +{\nu} \|\M\|_{L^2(\R^+\times\O)}^{{2}}.
$$
Doing as above gives that 
\begin{eqnarray*}
&&\sup_{t\in\R^+}\mathcal{E}_{BD}(\sro(t),\sro(t) u(t))+\frac{1}{4}\int_{\R^+}\mathcal{D}_{BD}\,dt+r_0\|u\|^2_{L^2(\R^+\times\O)}+r_1 \int_0^\infty\int_\O \rho |u|^4\,dx\,dt\\
&&\qquad\qquad \leq C_\nu (1+\kappa) (\|\M\|^2_{L^2(\R^+\times\O)}+\|f\|^2_{L^1(\R^+;L^2(\O))})+\mathcal{E}_{BD}(\sro_0,\sro_0 u_0) + E(\sro_0,\sro_0 u_0).
\end{eqnarray*}

\subsection{If $\nu r_1\leq 1/9$}  

We have 
\begin{eqnarray*}
&&\qquad \qquad \left| r_1\nu\int_0^\infty\int_\O 2\sro |u|^2u\cdot\nabla\sro\,dx\,dt\right|\leq 3\nu r_1  \int_0^\infty\int_\O \sro|u|^2\sro|\nabla u|\,dx\,dt\\
&&\leq \frac{\nu}{2}\int_0^\infty\int_\O\rho|\nabla u|^2\,dx+\dfrac{9}{2}\nu r_1\int_0^\infty\int_\O r_1\rho |u|^4\,dx\,dt .
\end{eqnarray*}

So considering $\mathcal{F}(t)=\mathcal{E}_{BD}+E$, we find directly:
$$
\partial_t\mathcal{F}(t)\leq C_\nu (1+\kappa)\|\M(t)\|^2_{L^2(\O)}+\|f(t)\|_{L^2(\O)} \sqrt{\mathcal{F}(t)}.
$$

We end the proof in the same way.

  \bibliography{biblioVV}
\bibliographystyle{plain}
\end{document}